\documentclass[12pt]{amsart}
\usepackage{amscd}
\usepackage{amsmath}
\usepackage{amsthm}
\usepackage{amsfonts}
\usepackage{amssymb}
\usepackage{pdfsync}

\newtheorem{pro}{Proposition}[section]
\newtheorem{proposition}[pro]{Proposition}
\newtheorem{lemma}[pro]{Lemma}
\newtheorem{theorem}[pro]{Theorem}

\newtheorem{corollary}[pro]{Corollary}
\newtheorem*{corollary*}{Corollary}
\newtheorem*{thm*}{Theorem}

{\theoremstyle{definition}
\newtheorem{definition}[pro]{Definition}
 \newtheorem*{example}{Example}

 }

\newcommand{\cM}{{\mathcal{M}}}

\newcommand{\cS}{{\mathcal{S}}}
\newcommand{\cO}{{\mathcal{O}}}

\newcommand{\Diff}{{\mathrm{Diff}}}
\newcommand{\ra}{{\, \rightarrow \, }}
\newcommand{\lra}{{\, \longrightarrow \, }}
\newcommand{\bbR}{\mathbb{R}}
\newcommand{\Isom}{\mathrm{Isom}}
\newcommand{\Aff}{\mathrm{Aff}}
\newcommand{\Aut}{\mathrm{Aut}}
\newcommand{\C}{\mathrm{Z}} 
\newcommand{\rank}{{\mathop{\mathrm{rank}\, }}}
\newcommand{\bbQ}{\mathbb{Q}}
\newcommand{\bbZ}{\mathbb{Z}}

\newcommand{\ncan}{\nabla^{\rm can}}
\newcommand{\Tau}{\mathcal{T}}

\def\N{{\Bbb N}}

\def\C{{\Bbb C}}

\def\cC{{\mathcal C}}
\def\cM{{\mathcal M}}
\def\eps{\varepsilon}

\def\ts{\thinspace}

\begin{document}

\author{Oliver Baues and  Wilderich Tuschmann}
\address{Oliver Baues and Wilderich Tuschmann
\\ Fakult\"at f\"ur Mathematik
\\ Karlsruher Institut f\"ur Technologie (KIT)
\\ D-76128 Karlsruhe, Germany}
\email{oliver.baues@kit.edu}
\email{tuschmann@kit.edu}

\title[Seifert fiberings and collapsing]{{Seifert fiberings and collapsing of infrasolv spaces}  
\\ \hspace{1cm} \newline
\emph{\tiny \today}}



\begin{abstract} We give a purely geometrical smooth characterization of
closed infrasolv manifolds and orbifolds 
by showing that, up to diffeomorphism, 
these are precisely the spaces which admit a collapse with 
bounded curvature and diameter to compact flat orbifolds.
Moreover, we distinguish irreducible smooth fake tori geometrically from standard ones
by proving that the former have non-vanishing $D$-minimal volume.
\end{abstract}

\maketitle

\section{Introduction}

In this article  we consider smooth orbifolds $\cO = X/ \Gamma$, 
where $X$ is a simply connected 
Riemannian manifold and $\Gamma$ is a
properly discontinuously acting group of isometries. 
We also assume that a finite orbifold covering space of $\cO$ is a manifold. 
If $X$
is contractible, then $\cO$ is called an \emph{aspherical orbifold}.
For the orbifolds considered here, the notion of smooth map and 
diffeomorphism can be declared by considering equivariant 
smooth maps on $X$. In what follows we are interested in
geometrical properties of orbifolds which distinguish
their diffeomorphism class. \\

Our first result relates collapsing of Riemannian 
orbifolds and the theory of Seifert fiberings: 

\begin{theorem} \label{thmA}
Let $M$ be a closed Riemannian orbifold which admits a bounded curvature collapse to a compact aspherical Riemannian orbifold $\cO$. Then $M$ admits the structure of
a smooth Seifert fibering over $\cO$ with infranil fiber. In particular, $M$ is an aspherical Riemannian orbifold.
\end{theorem}

Seifert fiber spaces with infranil fiber are a special class of orbibundles whose fibers are naturally diffeomorphic to infranil-manifolds. 
The notion of Seifert fiber space with infranil fiber  was systematically developed by Raymond, Lee and others. 
 (See \cite{LR_3, LR_4} for
recent accounts.  The precise definition of Seifert fiber spaces is also given in Section \ref{sect:Seifert} below.) \\

A Riemannian orbifold $\cS = X/\Gamma$ is called 
an {\em infrasolv orbifold} if $X$ is isometric to a simply
connected solvable Lie group $S$ with left-invariant metric
and $\Gamma$ is contained in the group of affine isometries
of $S$. It is called {\em infranil} if $S$ is nilpotent. Remarkably,
as is proved in \cite{Baues}, infrasolv orbifolds are diffeomorphic if and only if their fundamental groups are isomorphic. 

\begin{theorem}  \label{thmB}
Let $\cS$ be a closed infrasolv manifold (or orbifold) 
with fundamental group $\pi_{1}(\cS)$. Then $\cS$ admits a canonical Seifert fibering with  infranil fiber over a compact 
flat orbifold $\cO$. 
\end{theorem} 

Contracting the fibers of the above bundle to a point in an inhomogeneous way
(see, for example, \cite{Tusch}) shows that $\cS$ admits a bounded curvature collapse to the flat orbifold $\cO$. Moreover, in this setting $\cS$ is infranil if and only if it collapses to a point. \\

We also prove the following converse to Theorem \ref{thmB}: 

%

\begin{theorem} \label{thmC}
Assume that the closed manifold 
(or compact orbifold) $M$  allows a bounded curvature collapse to 
a compact infrasolv orbifold $\cO$.  Then $M$ is diffeomorphic to 
an infrasolv manifold (respectively, orbifold) $\cS$.
\end{theorem}


We thus obtain a purely geometrical smooth characterization of infrasolv 
manifolds 
as follows:

\begin{corollary} A closed manifold is diffeomorphic 
to an infrasolv manifold if and only if it admits a bounded
curvature collapse to a compact flat orbifold. 
\end{corollary}  

The corollary 
generalizes the famous almost flat manifold theorem of Gromov-Ruh \cite{Gro_78,Ruh_82} from infranil to infrasolv manifolds. It also holds for orbi- instead of manifolds (cf.\ Section 4), and strengthens the \emph{topological} characterizations of infranil and infrasolv manifolds in terms of collapsing, which were developed in 
 \cite{F-H}  and in \cite{Tusch}, to the smooth case. \\

The proof of the preceding results relies on techniques in the theory of collapsing of Riemannian manifolds (see the appendix of this article) and, as another crucial ingredient, on the rigidity of smooth Seifert fiberings with infranil fiber as established 
in \cite{KLR}. We recall the required rigidity results in Section \ref{sect:Seifert_rigidity}.\\

The following applications shed some light on the geometry of fake tori. 
Recall that a smooth fake torus $\Tau^n$ is a smooth manifold which is homeomorphic but not diffeomorphic to a standard torus $T^n$. 
The existence of smooth fake tori, $n \geq 5$, was established
by Wall and Browder, see \cite{Browder}, \cite[15 A]{Wall}.
Since, 
as follows from \cite{Baues}
,  smooth fake tori do not carry an infrasolv structure, we first have: 

\begin{corollary} A smooth fake torus does not 
allow a bounded curvature collapse to a compact flat orbifold. 
\end{corollary}  

\medskip 
Gromov has defined the {\it minimal volume} of a smooth
manifold $M$, $\text{MinVol}(M)$, as the infimum of all volumes $vol_g(M)$,
where $g$ ranges over all smooth complete Riemannnian metrics on $M$ 
whose sectional curvature is bounded in absolute value by one.
Gromov's critical volume conjecture asserts that there exists 
$\delta(n)>0$
such that if a closed smooth $n$-manifold $M^n$ 
admits a metric with
$\vert \text{sec} (M,g) \vert \le 1$
and $\text{vol}(M,g)\le\delta(n)$,
then $\text{MinVol}(M)=0$. For $n\le 4$ 
this conjecture is known to hold, see \cite{R_93, CR_96} and the further references cited there.

Given a real number $D>0$, one may also consider the {\it D-minimal volume}, $\text{$D$-MinVol}(M)$, where one 
requires that the infimum is taken over all  metrics $g$ 
as above,  which additionally have diameter bounded from
above by $D$.  
Cheeger and Rong have shown that there exists $\delta=\delta(n,D)>0$
such that if a closed smooth $n$-manifold $M$ admits a Riemannian metric $g$
with $\vert \text{sec} (M,g) \vert \le 1$, $\text{diam}(M,g) \le D$, 
and $\text{vol}(M,g)\le\delta(n,D)$,
then $\text{$D$-MinVol}(M)=0$, see \cite{CR_96}.
In the special context of aspherical $M$ this was
proved by Fukaya in \cite{Fukaya_90}.

\ 

Call a fake smooth torus \emph{irreducible} if it is not a product 
of a standard torus and a fake torus of lower dimension. 
For example, smooth fake tori obtained by taking the connected sum of standard tori
with exotic spheres are always irreducible. Then the following holds:

\begin{corollary} \label{cor:irreducibletori}
Let $\Tau$ be a smooth fake torus which
is irreducible. Then, for all $D>0$, 
there exists a positive constant $\nu(D)$
such that
$$ \text{$D$-$\mathrm{MinVol}$}  (\Tau)  \, \geq \, \nu(D)>0\; . $$
\end{corollary}

Notice that closed infrasolvmanifolds have vanishing
minimal volume, since, by Corollary 1.4, they actually have vanishing
D-minimal volume for appropiate $D>0$.

\

\noindent
{\bf Question 1.7.}
{\it Do smooth fake irreducible tori also always have nonvanishing minimal volume ?

}

%

\section{Seifert fiber spaces with nil-geometry}
\label{sect:Seifertmain}

\subsection{Infra-$G$ manifolds} \label{sect:infraN}
Let $G$ be a Lie group. We let $\Aff(G) = G \cdot \Aut(G)$ 
denote the group of affine transformations of $G$.
\emph{Note that the
affine group $\Aff(G)$ is precisely the normalizer of (the left-action of) $G$ in the group of all diffeomorphisms $\Diff(G)$.}

\begin{definition} Let $\Delta \leq
\Aff(G)$ be a discrete subgroup whose homomorphic image in 
$\Aut(G)$ has compact closure. The quotient $G/\Delta$ is called an \emph{infra-$G$ space}. If $G/\Delta$ is a manifold then
it is called an \emph{infra-$G$ manifold}.
\end{definition}

\begin{example}[Infranil manifolds] 
Let $N$ be a simply connected nilpotent Lie group. Then 
a  compact  infra-$N$ manifold $N/\Delta$ is called an 
\emph{infranil manifold}. For an infranil manifold $N/\Delta$,
 the intersection
$\Delta_{0} = N \cap \Delta$ is the maximal nilpotent normal
subgroup of $\Delta$ and it has finite index in $\Delta$. (Bieberbach's first theorem, see \cite[Chapter VIII]{Raghunathan}.) 
In particular, for every infranil manifold, $\Delta$ has finite image in $\Aut(N)$, and $N/\Delta$
 has a canonical finite normal covering by the \emph{nilmanifold} $N/ \Delta_{0}$.
\end{example}

\subsubsection{Affine group} 
Let $\nabla^{\rm can}$ denote the  natural flat 
connection on $G$  which is defined by the property that left-invariant vector fields are parallel.
The group of connection preserving diffeomorphisms 
$\Aff(G,\nabla^{\rm can})$ then 
coincides with the group of affine transformations
$ \Aff(G)$ of $G$.
In particular, every infra-$G$ manifold $$ G/\Delta $$  
carries a flat connection $\nabla^{\rm can}$
induced from $G$. The group of connection preserving diffeomorphims of $(G/\Delta, \nabla^{\rm can})$ arises from the normalizer $N_{\Aff(G)}(\Delta)$ of $\Delta$ in the group $\Aff(G)$ 
as 
$$  \Aff(G/\Delta, \nabla^{\rm can}) \cong   N_{\Aff(G)}(\Delta) / \Delta \, . $$ 
(We will also use the shorthand 
 $ \Aff(G/\Delta)=  \Aff(G/\Delta, \nabla^{\rm can})$.) 

\subsection{Infra-$N$ bundles}
Let $f: M \ra B$ be a (locally trivial) fibration of smooth manifolds 
with fiber $F$.
\begin{definition} \label{def:infraNb}
The fibration $f$ is called an \emph{infra-$G$ bundle} over $B$ if $F = G/\Delta$ is an infra-$G$ manifold 
and the structure group of $f$ is contained in the affine group $\Aff(G/\Delta)$. 
\end{definition}

\subsubsection{Fiberwise affine maps}
Let $f: M \ra B$ and $f':M' \ra B'$ be infra-$G$ bundles with fiber
$G/\Delta$. Then a map $$ \varphi: M \,  \lra \, M'$$ is called
\emph{fiberwise affine} if it is a morphism of $\Aff(G/\Delta)$-bundles (that is, equivalently 
$\varphi$ is a bundle map which induces an affine diffeomorphism 
on each fiber). We let $\Aff(M,G/\Delta) = \Aff(M,G/\Delta,f)$ denote the group of all fiberwise affine diffeomorphisms of $f$. 

\subsection{Seifert fiber spaces} \label{sect:Seifert}
%
The following notion of Seifert fiber space 
has been  developed in \cite{KLR},
see also \cite{LR_3}.

\subsubsection{Construction}
Let $N$ be a Lie group which acts properly and freely on 
the manifold $X$,
and put $W = X/N$. We assume that $W$ is 
simply connected.
The normalizer of $N$ in $\Diff(X)$ will
be denoted by $\Diff(X,N)$. 
Let $\pi$ be  a group and 
$$ \rho: \pi \, \lra \,   \Diff(X,N)$$  an action of $\pi$
which is properly discontinuous. 
Put $\pi_{N} = \rho^{-1}(N)$ for the subgroup
of all elements in $\pi$ which act on $X$ via translations
of $N \leq \Diff(X)$. 

\begin{definition}  \label{def:Seifert}
Data $(X,N,\pi)$ as above define a \emph{Seifert fiber space}
if the following conditions are satisfied:
\begin{enumerate}
\item  
$\Gamma = \rho(\pi) \cap N$ is a discrete uniform subgroup of $N$,
\item the induced action of $\Theta = \pi\big/\pi_{N}$ on $W$
is properly discontinuous.
\end{enumerate}
\end{definition}

In the situation of Definition \ref{def:Seifert},  $\rho$ is called a \emph{Seifert action}. If the action of $\pi$ is faithful then it 
gives rise to an exact sequence of groups 
\begin{equation*}  \label{eq:ext0}
1 \lra \Gamma \lra \pi \lra \Theta \lra 1 \; \; , 
 \end{equation*}
where $\Gamma = \pi_{N}$ is isomorphic to a lattice
in $N$. 
\emph{Note that, in general, the induced action 
of $\Theta$ on $W$ may have a finite kernel.}

\subsubsection{Seifert bundle maps}
A Seifert fiber space gives rise to a Seifert bundle map
$$  \sigma: X\big/\pi \, \longrightarrow \,  W\big/ \pi = W\big/ \Theta \;  $$
whose fibers are compact infra-$N$ spaces. The 
space $X\big/\pi$ is called a \emph{Seifert bundle} over the base orbifold $W\big/\Theta$ with typical fiber the 
$N$-manifold $N/\rho(\pi) \cap {N}$. \\
%

\paragraph{\em Remark}%
If $X/\pi$ is a manifold then, in fact, all fibers 
of $\sigma$ are (compact) infra-$N$ manifolds. If also 
the base $W/\pi$ is a manifold then all fibers of $\sigma$ 
are naturally diffeomorphic to $N/\Delta$, where $\Delta \leq \pi$ 
is the normal subgroup of $\pi$ which acts trivially on $W$.
In this situation, since $\pi$ normalizes $N$ and $\Delta$, the Seifert bundle map $\sigma$ defines a fibration whose 
structure group is discrete and contained in $\Aff(N/\Delta)$.
In particular, $\sigma$ is an infra-$N$ bundle with fiber $N/\Delta$
in the sense of Definition \ref{def:infraNb}. In general, a Seifert fibering does \emph{not} give a locally trivial fiber bundle over the orbifold $W/\pi$, but induces the structure of an 
\emph{orbibundle} over $W/\pi$. See \cite{LR_3} for various examples.

\subsubsection{Seifert structures on manifolds and orbifolds}

\begin{definition} \label{def:seifertstructure}
Let $M$ be a smooth orbifold and $(X,N,\pi)$ a 
Seifert fiber space.
A \emph{Seifert fiber structure} on $M$ is a diffeomorphism 
$$ \varphi : X/\pi \; \lra \, M \,. $$ 
\end{definition}

\subsection{Seifert fiber spaces with nil-geometry} 
\label{sect:Seifert_rigidity}
In this subsection 
let $N$ denote a simply connected nilpotent Lie group.
A \emph{lattice} in $N$ is a discrete uniform subgroup
of $N$.

\subsubsection{Existence} 

Let $\Gamma$ be a finitely generated torsion-free 
nilpotent group and 
\begin{equation}  \label{eq:ext1}
1 \lra \Gamma \lra \pi \lra \Theta \lra 1 \end{equation}
an exact sequence of groups. 

\begin{definition} A Seifert fiber space 
$(X,N,\pi)$
is said to realize the group extension \eqref{eq:ext1} 
if $\rho(\Gamma)$ is a lattice of  $N$.  
\end{definition}

{\em Remark.} In this situation, $\Gamma$ is a subgroup
of finite index in $\pi_{N} = \rho^{-1}(N)$, and therefore the induced action of 
$\Theta$ on $W = X/N$ is  properly discontinuous.

\begin{theorem}[{\cite[\S 2.2]{KLR}}] Assume that $\Theta$ acts properly discontinuously on a simply connected manifold $W$ and let $N$ be a simply connected nilpotent Lie group which contains $\Gamma$ as a lattice. 
Then there exists a Seifert fiber space $(X,N,\pi)$,  
which realizes the group extension 
\eqref{eq:ext1} 
 and  which induces the given action of $\Theta$ on $W$. 
\end{theorem}

\subsubsection{Rigidity}
%
Let $N$ and $N'$ be simply connected nilpotent Lie groups,
$(X, N, \pi)$ and $(X', N', \pi')$ Seifert fiber spaces, and
$\phi: \pi \lra \pi'$ a homomorphism of groups.


\begin{definition} 
An equivariant diffeomorphism of actions $$ (f, \phi): (X, \pi)
\longrightarrow  (X', \pi')$$ is called an \emph{affine equivalence} of Seifert
actions if $ f N f^{-1} = N'$.
\end{definition}

%
Now consider an isomorphism of 
group extensions of the form 
\begin{equation}\label{eq:seifertiso}
\begin{CD}
1 @>>> \Gamma @>>> \pi @>>>  \Theta  @>>> 1\\
 @VVV   @VV{\phi_1}V @VV{\phi}V @VV{\bar \phi_{}}V @VVV \\
1 @>>> \Gamma'  @>>> \pi' @>>>  \Theta' @>>> 1
\end{CD} \; \, \; \;  \;  .
\end{equation}

\hspace{1cm}\\
The Seifert rigidity for nil-\-geometry
may be stated as follows:

\begin{theorem}[{\cite[\S 2.4]{KLR}}] \label{thm:nilrigidity} 
Suppose $(X,N, \pi)$ and $(X',N',\pi')$ are Seifert fiber spaces
which realize the group extensions in diagram \eqref{eq:seifertiso}.
If 
$$(\bar{f}, \bar \phi) : (W, \Theta)
\ra (W', \Theta') \; ,  $$  
is 
an equivariant diffeomorphism 
then there exists a lift of $\bar f$ 
to an affine equivalence of Seifert actions %
$$
 (f, \phi):   ( X, \pi)    \ra   (X', \pi')  \, .
$$
%
\end{theorem}

Note that, as a
consequence, the Seifert rigidity constructs a
diffeomorphism of Seifert fiber spaces
$$ X/\pi \to X'/\pi' \; , $$ 
which restricts to affine (in particular, connection preserving) maps
on the infranil fibers.  

\subsection{Infranil bundles over aspherical manifolds}
\label{sect:infranilbundles}
Let $N$ be a  simply connected nilpotent Lie group.
In this subsection we relate the notions of infra-$N$
bundle and Seifert fiber space more closely. 
Recall that a manifold is called \emph{aspherical}
 if its universal covering
manifold is contractible.  
The following result asserts that over an
aspherical base the structure group of an infra-$N$ bundle
may be reduced to a discrete group. This implies
that \emph{infranil bundles over an aspherical base are
fiberwise affinely equivalent to a Seifert fiber space}.

%
%

\begin{proposition} \label{prop:isSeifert}
Let $f: M \ra B$ be an infra-$N$ bundle, where $B$ is 
an aspherical manifold and $N$ is a  simply connected nilpotent Lie group. Let $\pi = \pi_{1}(M)$ and assume also that $f$ 
has compact fiber $N/\Delta$.  Then there exists 
a Seifert fiber space $(X, N, \pi)$ which admits
a fiberwise affine diffeomorphism $$ \varphi: X/\pi \lra M$$  
to the infra-$N$ bundle $M$.
\end{proposition} 
%
%
\begin{proof} By the long exact homotopy sequence of  the
fibration $f$ we have an exact sequence 
$$  1 \ra \Delta \ra \pi_{1}(M) \ra \Theta \ra 1$$
where $\Theta \cong \pi_{1}(B)$. 
Let $\mathsf{p}: W \ra B$ be the universal covering of $B$,
and $\bar f: M' \ra W$ the induced infra-$N$ bundle. 
The induced morphism of bundles $M' \ra M$ is a covering; its group of deck transformations is isomorphic to $\Theta$ and it 
acts by fiberwise affine maps on $\bar f: M' \ra W$. 

Since $W$ is contractible, the covering homotopy theorem
implies that $\bar f: M' \ra W$ is $G$-equivalent to 
the product bundle $N/\Delta \times W$, where $G = \Aff(N/\Delta)$ is the structure group of $\bar f$. That is,
there exists a  fiberwise affine diffeomorphism 
$$  N/\Delta \times W \to M' $$ of bundles over 
$W$. 
This map lifts to a diffeomorphism 
$$ \tilde \varphi: X \to  \tilde M  \,   $$
of $X =  N \times W$ to the universal cover 
$\tilde M$ of $M$. Via $\tilde \varphi$, 
the group $\pi = \pi_{1}(M)$ has an induced  
fiberwise affine action on the (trivial) $N$-principal 
bundle $X= N \times W$. 
Let $\Delta_{0}$ be the maximal
nilpotent normal subgroup of $\Delta$. Then 
$\Delta_{0}$  is normal in $\pi$, and 
also $\Delta_{0} = \Delta \cap N$ is a lattice in $N$. 
Since $\pi$ normalizes the lattice $\Delta_{0}$, the action of $\pi$ on $X$ being 
fiberwise affine, Lemma \ref{lem:normalizers} implies that $\pi$ normalizes $N$, that is, $\pi \leq \Diff(N \times W,N)$. Therefore, the data $(N \times W, N, \pi)$ define a  Seifert fiber space which is 
affinely equivalent to the infra-$N$ bundle $M$. 
\end{proof}

We remark:
\begin{lemma} \label{lem:normalizers}
Let $(X,N)$ be a principal bundle and $\Delta_{0} \leq N$ a lattice in $N$. Let $N_{\Diff(X)}(\Delta_{0})$ denote the normalizer of the left action of $\Delta_{0}$ in $\Diff(X)$. 
Then $$  \Aff(X,N) \cap N_{\Diff(X)}(\Delta_{0}) \, \leq \; \Diff(X,N) \; . $$
\end{lemma}
\begin{proof} Using local trivializations we may assume that $X = N \times W$ is a trivial bundle. Every element $\gamma$ of $\Aff(X,N)$, being a fiberwise affine diffeomorphism of $(X,N)$, gives rise to a 
family of automorphisms $\ell_{w} \in \Aut(N)$, $w \in W$,
which is defined by the property that $$ \gamma n \gamma^{-1} (v, w) = (\ell_{w}(n) \, v, \bar \gamma w) \; \, \;\text{,  for all $n,v \in N$} . $$
(Here $\bar \gamma: W \ra W$ denotes the map which $\gamma$ induces on $W$.)  If $\gamma$ normalizes 
$\Delta_{0}$ then, by Malcev rigidity of the lattice $\Delta_{0}$, $\ell_{w} = \ell$ 
is constant, where $\ell \in \Aut(N)$ is the unique extension of 
the automorphism which $\gamma$ induces on $\Delta_{0}$.
Therefore $\gamma$ normalizes $N$.  
\end{proof}

In our context, Seifert fiber spaces may be viewed as infra-$N$ bundles with singular fibers: 

\begin{proposition} \label{prop:QuotisSeifert}
Let $f: M_{0} \ra B$ be an infra-$N$ bundle over the aspherical base manifold $B$, where $N$ is a  simply connected nilpotent Lie group. Assume that $f$ has  compact fiber $N/\Delta$ and let
$\mu \leq \Aff(M_{0}, N/\Delta)$ be a finite group of fiberwise affine diffeomorphisms. Put $$ M = M_{0} / \mu \; . $$
Let $\pi  = \pi_{1}(M)$ be the orbifold fundamental group of 
$M$ and $\pi_{0} = \pi_{1}(M_{0})$.  
Then there exists 
a Seifert fiber space $(X, N, \pi)$ over the base orbifold $B/\mu$, a Seifert fiber structure $ \varphi: X/\pi \lra M $ and a commutative
diagram \begin{equation*} \label{eq:quottiso}
\begin{CD}
X/\pi_{0} @>>{\varphi_{0}}> M_{0} @>>> B\\
 @VVV   @VVV  @VVV\\
X/\pi  @>>{\varphi}> M  @>>> B/\mu
\end{CD} \; \, \; \;  \;  ,
\end{equation*}
where $\varphi_{0}$ is a fiberwise affine diffeomorphism of bundles over $B$.
\end{proposition} 
\begin{proof}
Let $\pi = \pi_1(M)$. By construction, the group $\pi$ acts as a group of 
decktransformations of the universal covering 
$\tilde M  \ra M$ 
on $\tilde M$. This extends the corresponding action of
$\pi_{0}$, and also induces the action of $\mu$ on 
$M_{0} = \tilde M  / \pi_{0}$. Since the fibration $f$ is $\mu$-equivariant,
the decktransformation group $\Delta \leq \pi_{0}$ for the fiber of $f$ is normalized by $\pi$, and also the maximal nilpotent normal subgroup $\Delta_{0}$ is normalized by $\pi$.
Recall that $\Delta_{0}$ is a lattice in $N$. 

Proposition \ref{prop:isSeifert} constructs a Seifert fiber space 
$(X,N, \pi_{0})$ and a fiberwise affine 
diffeomorphism $\varphi_{0}$ as required. 
Let $ \tilde \varphi :  (X,N) \ra \tilde M$
be a lift of the diffeomorphism $\varphi_{0}$ to the universal covers. Then $ \tilde \varphi$ is an affine map of bundles. 
We may pull back the action of $\pi$ to an action on $X$. 
Since by iii) above, $\mu$ 
acts on the fibers of $f$ by affine transformations, 
the transported action of $\pi$ on $X$ 
is fiberwise affine for the bundle $(X, N)$.
That is, $\pi$ is contained in $\Aff(X,N)$.
Using Lemma \ref{lem:normalizers}, it follows that, a fortiori,  $ \pi$ 
is contained in $\Diff(X,N)$. Therefore, the
data $(X,N,\pi)$ satisfy the axioms of a Seifert action.
\end{proof}

\section{Collapsing and Seifert fiberings}


We consider a compact aspherical orbifold  $\cO$, where  
$$ {\cO} = Y / \Theta \; , $$  $Y$ is a contractible 
Riemannian manifold and $\Theta \leq \Isom(Y)$ acts 
properly discontinuously on $Y$. 
Assume further that $(M,g)$ is a closed Riemannian orbifold which \emph{collapses} to $\cO$ with bounded curvatures.  
This means that there
exists a sequence of Riemannian metrics $g_{i}$ on
$M$ with uniformly bounded sectional curvature 
such that the metric spaces $M_{i} = (M, g_{i})$
converge to $\cO$ in the Gromov-Hausdorff distance.
That is, we have 
$$ \lim_{i \to \infty} d_{\mathrm GH}(M_{i},\cO) = 0 \;  .
$$
For the Gromov-Hausdorff distance and its generalizations
see Appendix~A.

In this section,  we show 
that the orbifold $M$ carries a Seifert fiber structure 
over the base $\cO$. 
This proves Theorem \ref{thmA} in the introduction. 
In Section \ref{sect:proofofit} we also 
prove Corollary \ref{cor:irreducibletori}.

%

\subsection{Equivariant collapsing on finite covers}
Let $\Theta_{0}$
be  a torsion-free normal subgroup of $\Theta$, which
is of finite index. By our general assumption on orbifolds, 
such a subgroup of $\Theta$ always exists. 
Then  $$ T =   Y/ \Theta_{0}$$ is a closed Riemannian manifold  
and it is an orbifold covering space of $\cO$. Note that 
$\cO$ is a quotient of $T$ by the 
group of decktransformations 
of the map $T \ra \cO$. This is
a finite group isomorphic to 
$ \Theta/\Theta_{0}$.

\begin{proposition}  \label{pro:liftconvergence}
With respect to an appropiate choice of $ \Theta_{0}$,
there exists a sequence of 
Riemannian manifolds 
$\bar M_{i} = (\bar M,g_{i})$, Riemannian orbifold covering maps 
$$ \bar M_{i} = (\bar M,g_{i})  \, \lra  \, (M, g_{i}) , $$ 
 and a finite group $\mu$ which acts
 by isometries on $\bar M_{i}$ and $T = Y/\Theta_{0}$ 
such that 
\begin{enumerate}
\item 
$M_{i} = \bar M_{i} /\mu$  
and  $\cO =  T/ \mu$; 
\item 
the equivariant Hausdorff distance 
with respect to $\mu$ satisfies 
$$
 \lim_{i \to \infty} d_{\mu-\mathrm{GH}}(\bar M_{i},T) = 0 \;  .
$$
\end{enumerate}
\end{proposition}
\begin{proof} Let $X_{i}$ denote the universal cover of $M_{i}$. Then $M_{i} = X_{i}/ \pi$, where 
$\pi$ acts on $X_{i}$ as a discrete group of
isometries.
By Proposition A.7 in Appendix~A, there exists a sequence 
$(X_{i}, \pi, p_{i}) \in \mathcal{M}^{Groups}$ such that  
$$ \lim_{i \ra \infty} (X_{i}, \pi, p_{i}) = (Y, \Theta, q) $$
with respect to the equivariant Gromov-Hausdorff convergence with base point. Note  that
this implies (compare Definitions A.2, A.3.) $$ \lim_{i \ra \infty} (X_{i}, p_{i}) = (Y, q) \, . $$ 
Now let $\pi'$ be a torsion-free finite index normal subgroup of $\pi$. Using Proposition A.4, we 
infer that there exists a closed subgroup $\Theta' \leq \Isom(Y)$ such that (up to taking a subsequence) 
$$ \lim_{i \ra \infty} (X_{i}, \pi', p_{i}) = (Y, \Theta', q) \; . $$
In particular, by Proposition A.6 and taking into account that the
involved spaces are compact, we have 
$$\lim_{i \ra \infty} X_{i}/\pi'  = Y/\Theta' \,. $$ 
Let $\mu' = \pi/ \pi'$. Then $\mu'$ acts isometrically as a finite group of decktransformations of the maps $X_{i}/ \pi' \ra X_{i}/ \pi$ and  furthermore $$ X_{i} / \pi \, = \, (X_{i}/ \pi') / \mu'  \; . $$

%

Let $\cM(\mu')$ denote the space of isometry 
classes of compact metric spaces with isometric action of 
$\mu'$. Since $\mu'$ is finite, the equivariant version of the pre-compactness theorem holds 
with respect to $\cM(\mu')$ (compare Proposition A.5 in Appendix A). 
Therefore, 
after taking a subsequence,  we may assume  that there exists 
$$    \lim_{i \ra \infty}  \left(X_{i}/ \pi', \mu' \right) \, = \, ( \cO', \mu'  ) \; \in \cM(\mu'). $$
This of course implies that 
$$  Y/ \Theta =   \lim_{i \ra \infty} X_{i}/ \pi =\lim_{i \ra \infty}  (X_{i}/ \pi') / \mu'   \; =   \; \cO' /  \mu'  \;  . $$
We deduce, in particular, that $\cO' =  Y/ \Theta'$ is a finite orbifold covering space of $Y/\Theta$ and  $\Theta'$ is a normal subgroup of $\Theta$. 
We have thus constructed a sequence of finite covering 
manifolds $M'_{i} = X_{i}/ \pi'$ of $M_{i}$,  which converges to an orbifold covering space $\cO' =  Y/ \Theta'$ of $\cO$ in the Gromov-Hausdorff-distance. 

Observe that the action of $\pi'$ on $X_{i}$ is free, since
$\pi'$ is torsion-free. Therefore, we may apply Proposition A.8 in Appendix~A to a torsion-free finite index normal subgroup 
$\Theta_{0}$ of $\Theta'$, and infer that there exists a subgroup $\pi_0$ of $\pi'$ such that 
\begin{enumerate}
\item $\pi_{0}$ is normal and of finite index in $\pi'$,
and furthermore 
\item 
$ \lim_{i \ra \infty} (X_{i}, \pi_{0}, p_{i}) = (Y, \Theta_{0}, q) \; . 
$
\end{enumerate}
This gives rise to a sequence of Riemannian manifolds  
$$\bar M_{i} = X_{i}/ 
\pi_{0}$$ converging  to the manifold $T = Y/  \Theta_{0}$ in 
the Gromov-Hausdorff distance, that is,  
$$   \lim_{i \ra \infty} \bar M_{i}   \, = \, T \; .  $$ 
Note that, by passing to
a finite index subgroup of $\pi_{0}$, and arguing as in the
first part of the proof, we can furthermore assume that 
\begin{enumerate}
\item[(3)] $\pi_{0}$ is normal in $\pi$.
\end{enumerate}
Now let $\mu = \pi/ \pi_{0}$. Then $\mu$ acts on $\bar M_{i}$ by isometries and  $M_{i} =  \bar M_{i}/ \mu$. 
Furthermore, there  exists a sequence 
$(\bar M_{i}, \mu ) \in \mathcal{M}(\mu)$ such that  
$$ \lim_{i \ra \infty} (\bar M_{i}, \mu) = (T, \mu) \; . $$
In other words, $\lim_{i \to \infty} d_{\mu-\mathrm{GH}}(\bar M_{i},T) = 0$. This also implies ordinary Gromov-Hausdorff convergence of $\bar M_{i}/ \mu$ to  $T/\mu$. Since 
$$  \lim_{i \ra \infty} \bar M_{i}/ \mu =  \lim_{i \ra \infty} M_{i} = \cO \; , $$ we conclude that $ \cO = T / \mu$. 
This proves the proposition. 
\end{proof}

Proposition \ref{pro:liftconvergence} 
allows to apply equivariant collapsing results
for Riemannian manifolds.
Indeed, 
Theorem~B.2 in Appendix~B
implies that  there exist a sequence of $\mu$-equivariant
fibrations $$ f_{i}: \bar M_{i} \ra T $$ such that, for $i$ sufficiently large,  \begin{itemize}
\item[i)] each fiber $f_{i}^{-1}(p)$ carries a flat connection (depending smooth\-ly on $p \in T$), such that $f_{i}^{-1}(p)$ is
affinely diffeomorphic to an infranil manifold 
$N_{i} /\Delta_{i}$ with canonical flat connection $\ncan$;
\item[ii)] the structure group of $f_{i}$ is contained in the
Lie group
$$  { \C(N_{i}) \over \C(N_{i}) \cap \Delta_{i} }  \rtimes {\rm Aut}(\Delta_{i})\,  ,
$$ where $\C(N_{i})$ denotes the center of $N_{i}$;
\item[iii)] The affine structures on the fibers of $f_{i}$ are $\mu$-invariant;
\item[iv)] $f_{i}$ is an almost Riemannian submersion.
\end{itemize} 
Note, in particular, that in the terminology of Section 
\ref{sect:Seifert}, by ii),  
$f_{i}$ is an infra-$N_{i}$ bundle over the base $T$ with fiber the
infranil manifold $N_{i} /\Delta_{i}$, and, by iii),  $\mu$ acts by fiberwise affine maps on the 
bundle $f_{i}$. 

\subsection{Induced Seifert fiber space structure}
Since $T$ is aspherical, on the level of fundamental groups 
the long exact homotopy sequence for 
the fibration $f_{i}: \bar M_{i} \ra T$ gives a short exact sequence
\begin{equation}
\label{eq:fib2}
  1 \lra \Delta_{i} \lra \pi_{0} \lra \Theta_{0} \lra 1  \; \;  , 
   \end{equation} 
where $\Delta_{i} \cong \pi_{1}(N_{i} /\Delta_{i})$ is (isomorphic to) a subgroup of
$\Aff(N_{i})$. Since the fiber $N_{i}/\Delta_{i}$ of $f_{i}$ is infranil, $\Delta_{i}$ 
intersects $N_{i}$ in a uniform discrete subgroup
$\Delta_{i0}$, which is of finite index in $\Delta_{i}$.
In particular, $\Delta_{i}$  is a torsion-free finitely generated virtually nilpotent group, and $\Delta_{i0}$ is its maximal nilpotent normal subgroup. Since $\Delta_{i0}$ is normal in $\pi_{0}$, we obtain another short exact sequence
\begin{equation}
\label{eq:fib20}
  1 \lra \Delta_{i0} \lra  \pi_{0} \lra  \bar \Theta_{0} \lra 1  \; \; ,
\end{equation} 
where $\bar \Theta_{0}$ acts properly discontinuously on $Y$
with quotient $T$. Proposition \ref{prop:isSeifert} and its proof now imply:
 
\begin{proposition} \label{pro:barSeifert}
The infra-$N_{i}$ bundle $f_{i}: \bar M_{i} \, \ra \, T$ together with its flat fiber connections arises from a Seifert fiber construction $(X_{i}, N_{i}, \pi_{0})$ for the exact sequence \eqref{eq:fib20} over the base $T$. 
\end{proposition} 

This means that, for all $i$, there exists a Seifert fiber space $(X_{i}, N_{i}, \pi_{0})$ over the base $T= Y/\bar \Theta_{0}$,  which realizes the group extension \eqref{eq:fib20},  and which has as (so called) typical fiber the nilmanifold $N_{i}/ \Delta_{0i}$. Moreover, there exists a diffeomorphism
$$ \bar \varphi_{i}: X_{i}/ \pi_{0} \,  \lra  \, \bar M_{i}$$ 
which induces  a
commutative diagram
\begin{equation}\label{eq:seifert_structure}
\begin{CD}
N_{i}/ \Delta_{i} @>{\approx}>> f_{i}^{-1}(p) \\
  @VVV  @VVV  \\
    X_{i}/ \pi_{0} @>{\approx}> {\bar \varphi_{i}}> \bar M_{i} @>>> \bar M_{i} / \mu = M_{i} \\  
   @VVV  @VV{f_{i}}V  @VVV\\   
   Y/ \pi_{0}  @>{\approx}>>  T @>>> T/\mu = \cO 
\end{CD} \; \, \; \;  \;  .
\end{equation}
such that the fiber maps 
$N_{i}/\Delta_{i} \stackrel{\approx}{\ra} f_{i}^{-1}(p)$ 
are affine diffeomorphisms. Since $f_{i}$ is $\mu$-equivariant,
it descends to the projection map $M_{i} \ra \cO$. 
\\

Let $\tilde \varphi_{i}: X_{i} \ra \tilde M_{i}$ be a lift
of $\bar \varphi_{i}$ to the universal covers.
Since by iii) above, $\mu$ 
acts on the fibers of $f_{i}$ by affine transformations, 
the transported action of $\pi$ on $X_{i}$ 
is fiberwise affine for the bundle $(X_{i}, N_{i})$.
That is, $\pi$ is contained in $\Aff(X_{i},N_{i})$.  
In the view of  Proposition \ref{prop:QuotisSeifert} we 
thus have:

\begin{proposition}  \label{pro:bundleM}
The Seifert action of  $ \pi_{0}$ on $(X_{i}, N_{i})$ extends to a Seifert action of $\pi$ on $(X_{i}, N_{i})$, such that the diffeomorphism $\tilde \varphi_{i}$ descends to a diffeomorphism 
of orbifolds $\varphi_{i}: X_{i}/\pi \ra   M_{i}$.
\end{proposition} 

Since $\Delta_{i}$ is a normal subgroup of $\pi$ and $\Delta_{i0}$
is characteristic in $\Delta_{i}$,  we have induced exact sequences of the form 
\begin{equation}
\label{eq:fibpiDeltai}
  1 \lra \Delta_{i0} \lra \pi  \lra \bar \Theta_{i} \lra 1  \; ;
 \end{equation}
moreover, for each $i$, there is a commutative diagram: \\

\begin{equation}\label{eq:}
\begin{CD}
1 @>>> \Delta_{i0} @>>>  \pi_{0} @>>>  \bar\Theta_{0}  @>>> 1\\ 
@VVV   @VV{\shortparallel}V @VVV @VVV @VVV \\
1 @>>> \Delta_{i0} @>>> \pi @>>>  \bar \Theta_{i}  @>>> 1\\
 @VVV   @VVV @VV{\shortparallel}V @VVV @VVV \\
1 @>>> \Delta  @>>> \pi @>>>  \Theta @>>> 1
\end{CD} \; \, \; \;  \;  .
\end{equation} 

\vspace{2ex}
Note that the quotient 
$\bar \Theta_{i}$ maps surjectively 
onto $\Theta$. Since, $\Delta_{i0}$ is of finite index in $\Delta$,
$\bar \Theta_{i}$ acts properly
discontinuously on $Y = X/ N_{i}$. As diagram \eqref{eq:seifert_structure} shows, the quotient space
$$ Y/ \bar \Theta_{i} = Y/ \Theta $$ is diffeomorphic to the 
orbifold $\cO = Y /\Theta$. Therefore,  the Seifert action 
$(X_{i}, N_{i},\pi)$ realizes the group extension \eqref{eq:fibpiDeltai} over the orbifold $\cO$.  
We arrive at:

\begin{theorem} There exists a diffeomorphism of the Seifert fiber space $X/\pi$ which realizes the group extension \eqref{eq:fibpiDeltai} over the base $\cO$ to the orbifold 
$M_{i}$.   
\end{theorem}

In particular, the smooth manifold (respectively orbifold) $M$ carries a Seifert fiber space structure over the orbifold 
$\cO = Y / \Theta$. 

\subsection{Collapsing of irreducible fake tori}  
\label{sect:proofofit}
We show now that irreducible fake smooth tori do not
collapse to aspherical Riemannian orbifolds.

\begin{proposition} Let $\Tau$ be an irreducible
fake smooth torus. Then $\Tau$ does not admit a collapse 
with bounded
curvature 
to a compact aspherical Riemannian orbifold
$\cO = Y/\Theta$.
\end{proposition}
\begin{proof} Assume that $M$ has such a collapse.
By Theorem \ref{thmA}, 
$M$ admits the structure of a Seifert fiber space over
$\cO$ with infranil fiber, which realizes a group extension
of the form 
$$ 1 \ra \bbZ^{n-k} \ra \bbZ^n   \ra \Theta \ra 1 \; . $$
Since $\Theta$ is abelian and
it acts properly
discontinuously with compact quotient on the contractible
manifold $Y$, we infer that $\Theta$ is torsion-free abelian,
isomorphic to $\bbZ^k$, $k < n$.
Therefore, the above extension is isomorphic to the direct product
$$    \bbZ^n =  \bbZ^{n-k} \times \bbZ^k \, .   $$ 
Note that $\cO = Y/ \Theta$ is a manifold and a homotopy torus. This implies that $\cO$ is also 
homeomorphic to a torus $T^k$. 
The Seifert rigidity Theorem \ref{thm:nilrigidity} implies
that the Seifert fiber space $\Tau$ over $\cO$ 
is affinely diffeomorphic to the product 
$T^{n-k} \times \cO$. Therefore, $\cO$ is 
a fake smooth torus. This contradicts the irreducibility
of $\Tau$. 
\end{proof}

We can give  now the proof for  Corollary \ref{cor:irreducibletori}: 
\begin{proof} 
Assume to the contrary that $D$-$\mathrm{MinVol} (\Tau) = 0$. Since $\Tau$ is aspherical,  this assumption 
implies (see the proof of Fukaya's Theorem \cite[Theorem 0-1]{Fukaya_88a} or \cite[Theorem 16.1]{Fukaya_90}) that $\Tau$ admits a bounded
curvature collapse to a compact aspherical Riemannian 
orbifold $\cO = Y/ \Theta$. This is not possible by the previous
proposition.
\end{proof}

\section{Seifert fiberings on infrasolv manifolds}
Let $S$ be a simply connected \emph{solvable} Lie group. 
\begin{definition}
A compact infra-$S$ space $S/\Delta$ will be called an \emph{infrasolv orbifold}. If $S/\Delta$ is a manifold, 
it is called an \emph{infrasolv manifold}. 
\end{definition}

In this section we 
prove that up to diffeomorphism every infrasolv orbifold arises from a Seifert fiber construction with nil-geometry over a flat Riemannian orbifold. In addition we show that the class of 
manifolds which are diffeomorphic to infrasolv manifolds is closed with respect to the Seifert fiber construction with nil-geometry. 
In view of Theorem \ref{thmA}, this implies Theorem \ref{thmC} 
in the introduction.

\subsection{Flat orbifolds associated to virtually abelian groups}
Let $\bar \Theta$ be
a  finitely generated virtually abelian group of rank 
$k$. Then there exists a 
homomorphism $\bar \Theta \ra \Isom(\bbR^k)$ such that
$\bar \Theta$ acts properly discontinuously on $\bbR^k$. 
This homomorphism is unique up to conjugacy by affine transformations (that is, by elements of $\Aff(\bbR^k))$.  (By Bieberbach's second and third theorems,
see, for example, \cite{Wolf}.) Let $\cO_{\bar \Theta} = \bbR^k /\bar \Theta$ 
denote the flat orbifold associated to $\bar \Theta$.

\subsection{Seifert fiberings over flat orbifolds}
The following implies Theorem \ref{thmB}
in the introduction:  

\begin{theorem} Let $M = S/ \pi$ be an infrasolv 
orbifold, and 
\begin{equation} 
\label{eq:fib4}
1 \lra \Delta_{0} \lra \pi \lra \bar \Theta \lra 1
\end{equation}  
an exact sequence of groups, where
$\Delta_{0}$ is a finitely generated torsion-free nilpotent group. 
Assume further that $\bar \Theta$ is
a virtually abelian finitely generated group of rank $k$. 
Then $M$ admits a Seifert fiber structure over 
the flat orbifold $\cO_{\bar \Theta}=\bbR^k \big/\bar \Theta$, which realizes
the group extension \eqref{eq:fib4}. 
\end{theorem}
\begin{proof} Since $M$ is diffeomorphic to a standard $\pi$-orbifold $M_{\pi}$ (see below), it is enough to show the
statement for $M_{\pi}$. Therefore, the theorem is a consequence of Proposition \ref{prop:standard_seifert} below.
\end{proof}

\begin{corollary} Every infrasolv orbifold has a Seifert fiber structure (with nil-geometry) over a flat orbifold. 
\end{corollary}
\begin{proof}
Indeed, we may  consider the Fitting subgroup $\Delta_{0}$
of $\pi$, which is (by definition) the maximal nilpotent normal subgroup of $\pi$.
This gives rise to an exact sequence of the form \eqref{eq:fib4}.
\end{proof}

Let $\pi$ be a torsion-free virtually polycyclic group.
Then (by \cite{Ausl_Johns}, but see also Section  \ref{sect:smooth_model} below),
there exists an infrasolv manifold $M_{\pi}$
with fundamental group $\pi$. 

\begin{corollary} Let $X/\pi$ be the total space of a
Seifert fiber space with nil-geometry over a flat orbifold 
$\cO=  \bbR^k \big/\Theta$. Then $X/\pi$ is (diffeomorphic to)
an infrasolv manifold.
\end{corollary}
\begin{proof} Since $X/\pi$ is a Seifert fiber space over the base $\cO$, there exists a homomorphism $\pi \ra \Theta$, as in \eqref{eq:fib4}. Let $M_{\pi}$ be an infrasolv manifold with fundamental group $\pi$, and $f:M_{\pi} \ra \cO$ the bundle map for the Seifert
fiber structure on $M_{\pi}$ which realizes the group extension representing the map $\pi \ra \Theta$.
By Theorem \ref{thm:nilrigidity}  (uniqueness of the Seifert construction), the infra-$N$ bundles 
$M_{\pi}$ and $X/\pi$
are fiberwise affinely diffeomorphic over $\cO$. 
In particular, $M_{\pi}$ and $X/\pi$ are diffeomorphic. 
\end{proof}

%
%
 \subsection{Flat affine manifolds}
 Let $U$ be a unipotent real linear algebraic group. Then $U$ is in particular a simply connected nilpotent Lie group. Conversely, every simply connected
nilpotent Lie group admits the structure of a unipotent 
real linear algebraic group. (Cf.\  \cite[Preliminaries, \S 2]{Raghunathan} for basic facts on linear algebraic groups.)
 \subsubsection{Flat manifolds modeled on $U$} \label{sect:Umanifolds}
We consider manifolds which
carry locally the \emph{affine} geometry of $U$. This generalizes the notion of \mbox{infra-$U$} manifold, which is an essentially
 metric concept.
 
Let $\Aff(U) = \Aff(U,\ncan)$ be the affine group of $U$.

\begin{definition} 
Let $\Delta \leq \Aff(U)$ be a discrete subgroup which acts properly 
on $U$ with compact quotient. Then the quotient $$ U /\Delta$$ is called a \emph{flat affine $U$-orbifold}. 
If $\Delta$ acts freely on $U$ then $ U /\Delta$ is a
 \emph{flat affine $U$-manifold}.
 \end{definition}

\emph{Remark.}
A long standing conjecture (``Auslander's conjecture'') 
asserts that $\Delta$ must be a virtually polycyclic group. 

\subsubsection{Unipotent shadow and algebraic hull} 
Now let $\Delta$ be a virtually polycyclic group without a finite (non-trivial) 
normal subgroup. By a construction, originally due to Mostow and Auslander 
(see \cite[Chapter IV, Proposition 4.4]{Raghunathan} and \cite[Appendix A.1]{Baues}, for reference), we may associate to $\Delta$ in a unique way a unipotent real linear algebraic group $U_{\Delta}$, which is called the
\emph{unipotent shadow} of  $\Delta$. 
The unipotent shadow $U_{\Delta}$ has the following
characterizing additional properties:
\begin{enumerate}
\item $\dim U_{\Delta} =
\rank \Delta$. (Recall that the rank of $\Delta$ coincides with the virtual cohomological dimension of $\Delta$.)
\item There exists a $\bbQ$-defined real linear algebraic
group $H= H_{\Delta}$ which has $\Delta \leq H(\bbQ)$ as a Zariski-dense subgroup. 
\item 
There is a splitting of algebraic groups $H = U_{\Delta} \rtimes T$, such that 
$T$ is a subgroup of semi-simple elements
which acts faithfully on $U_{\Delta}$.
\end{enumerate}
The group $H$ is called the \emph{algebraic hull of $\Delta$}.   

\subsubsection{Standard $\Delta$-manifolds} 
Since $T$ acts faithfully on $U_{\Delta}$, we may realize the hull $H = H_{\Delta}$ as a subgroup
of $\Aff(U_{\Delta})$ in a natural way. In particular, in this way, 
$\Delta$ acts faithfully on $U_{\Delta}$ by affine transformations.
Then we have:

\begin{theorem}[cf.\ \mbox{\cite[Theorem 1.2]{Baues}}]
The affine action of $\Delta$ on $U_{\Delta}$ is properly discontinuous,
and it is uniquely defined up to conjugacy with an element of the affine group 
$\Aff(U)$.
\end{theorem} 

The quotient space 
$$ M_{\Delta} \, = \, U_{\Delta} \big/ \Delta $$ is called a 
\emph{standard $\Delta$-orbifold.} If 
$ M_{\Delta}$ is a manifold it  is called a \emph{standard $\Delta$-manifold.} Remarkably, in standard $\Delta$-manifolds all homotopy equivalences of maps are induced by affine automorphisms of $U_{\Delta}$, see \cite{Baues}. 

\emph{Remark.}  If $\bar \Delta$ is an arbitrary virtually polycyclic group, dividing
by its maximal finite normal subgroup yields a unique homomorphism 
$\bar \Delta \ra \Delta$ onto the fundamental group of a standard $\Delta$-orbifold.
We also call $M_{\bar \Delta} = M_{\Delta}$ the standard $\bar \Delta$-orbifold associated to $\bar \Delta$. 

\subsubsection{Smooth rigidity of infrasolv manifolds} \label{sect:smooth_model}
As remarked in \cite{Baues}, every standard $\Delta$-manifold admits a (in general, non-unique) infrasolv manifold structure modeled on some solvable Lie group. In particular,  every torsion-free group $\Delta$ as above appears as the fundamental group of an infrasolv manifold.  Moreover, standard $\Delta$-manifolds may serve as \emph{unique} smooth models for infrasolv manifolds: 

\begin{theorem}[cf.\ \mbox{\cite[Theorem 1.4]{Baues}}]
Let $M = S / \Delta$ be a (compact) infrasolv manifold, where $S$ is a simply connected solvable Lie group and $\Delta \leq \Aff(S)$
a discrete group of isometries. Then $M$ is smoothly diffeomorphic to the standard $\Delta$-manifold $M_{\Delta}$.
\end{theorem} 

The corresponding statements  for infrasolv orbifolds and
standard $\Delta$-orbifolds hold as well, with the same proof
as in \cite{Baues}.

\subsection{Seifert fiberings on standard-manifolds}

Let $M_{\pi}$ be a standard $\pi$-orbifold, where $\pi$ is 
a virtually polycyclic group. Consider 
an exact sequence of groups
\begin{equation}
\label{eq:fib5}
1 \lra \Delta_{0} \lra \pi \lra \bar \Theta \lra 1 \; , 
\end{equation} where
$\Delta_{0}$ is a finitely generated torsion-free nilpotent group. 

\begin{proposition} \label{prop:standard_seifert}
There exists a Seifert fiber structure for $M_{\pi}$, which realizes the group
extension \eqref{eq:fib5} over the standard $\bar \Theta$-orbifold $M_{\bar \Theta}$.
\end{proposition}
\begin{proof} Let $U= U_{\pi}$ be the unipotent shadow of $\pi$. 
Recall that for the standard affine embedding $\pi \leq \Aff(U)$
the intersection $U \cap \pi$ is the maximal nilpotent normal subgroup
${\rm Fitt}(\pi)$ of $\pi$ (see \cite{Baues}). Since $\Delta_{0}$ is nilpotent, 
we therefore have $\Delta_{0} \leq U$. Let $U_{0}$ be the Zariski-closure of 
$\Delta_{0}$ in $U$. Then $\Delta_{0}$ is a finite index subgroup of
$\Delta_{0}'  = U_{0} \cap \pi$. Moreover, $U_{0}$ is normalized by
$\pi$, since $\Delta_{0}$ is a normal subgroup of $\pi$. 
Since $\pi$ is Zariski-dense in $H = U T$ it follows, in particular,
that $U_{0}$ is normal in $H$ and hence also in $U$.  
Let $U_{1} = U/U_{0}$ be
the quotient unipotent $\bbQ$-defined real linear algebraic group.

We consider now the induced homomorphism 
$H \ra \Aff(U_{1})$. Let $\Theta$ denote the 
image of $\pi$ in $\Aff(U_{1})$. It clearly is a Zariski 
dense subgroup of the image  $H_{1}$ of $H$
in $\Aff(U_{1})$.  Since $U_{1}$ is the unipotent
radical of $H_{1}$, we have the relation $\dim U -\dim U_{0} 
= \dim U_{1} \leq \rank \Theta$ (see \cite[Chapter IV, Lemma 4.36]{Raghunathan}). Since $\Delta_{0}$ is contained in the kernel of
$\pi \ra \Theta$, we have  $\rank \Theta  \leq \rank \pi - \rank \Delta_{0} = \dim U -\dim U_{0}$. We conclude that $\rank \Theta = \dim U_{1}$. It is now evident that the image  $H_{1}$ of $H$
in $\Aff(U_{1})$ satisfies the axioms for an algebraic hull of $\Theta$. Therefore
$M_{\Theta} = U_{1}/ \Theta$ is a standard $\Theta$-manifold. 
The data $(X, N, \pi) = (U, U_{0}, \pi)$ define a Seifert bundle
map $$ \sigma:  M_{\pi}   \ra M_{\Theta} $$
and the Seifert
construction $(U, U_{0}, \pi)$ realizes the group
extension \eqref{eq:fib5} over $M_{\bar \Theta} = M_{\Theta}$.
\end{proof}

\begin{corollary} The class of smooth orbifolds represented by
standard $\Delta$-orbifolds is closed with respect to the Seifert
fiber construction with nil-geometry.
\end{corollary}
\begin{proof} Indeed, let $\sigma: M  \ra M_{\bar \Theta}$ be a Seifert fiber 
space over a standard orbifold $M_{\bar \Theta}$, which realizes a group
extension of the form \eqref{eq:fib5}. By the previous proposition, 
the standard-$\pi$ orbifold $M_{\pi}$ also supports a Seifert construction
with nil-geometry over $M_{\bar \Theta}$, which realizes  \eqref{eq:fib5}.
By the Seifert rigidity theorem (Theorem \ref{thm:nilrigidity}), $M$
and $M_{\pi}$ are diffeomorphic.
\end{proof}

\section{APPENDIX}

\appendix

In the following sections we gather several important notions and results
from equivariant Gromov-Hausdorff convergence, developed by Fukaya and Yamaguchi,
and the work of Cheeger, Fukaya, and Gromov on collapsed Riemannian
manifolds with bounded sectional curvature. General references for these
topics are \cite{Fukaya_90} and \cite{Gro_99}.

\

\section{(Equivariant) Gromov-Hausdorff Convergence}
\noindent
Let us first recall the definition of the classical Hausdorff distance:
For subsets $A$ and $B$
of a metric space $X$, the {\it Hausdorff distance} between $A$ and $B$ in $X$, $d_H^X(A,B)$,
is defined as the infimum of all positive real numbers $\eps$ such that the open $\eps$-neighbourhood of $A$ in $X$ contains $B$ and vice versa. 
For any metric space $X$, the Hausdorff distance $d_H^X$ then
defines a metric on the set of
all closed and bounded nonempty subsets of $X$. 

Around 1980, Gromov introduced an abstract
version of the Hausdorff distance as follows:
The {\it Gromov-Hausdorff distance} of two compact metric spaces $X$ and $Y$, 
$d_{GH}(X,Y)$, is defined as the infimum of all numbers
$d_H^Z(f(X), g(Y))$, where $Z$ ranges over all metric spaces 
in which $X$ and $Y$ can be imbedded isometrically, and $f$ and $g$ over all isometric
embeddings $X\to Z$ and $Y\to Z$. Alternatively,
the Gromov-Hausdorff distance of $X$ and $Y$ 
can also be defined as the infimum of all Hausdorff distances
$d_H^Z(X,Y)$, where $Z$ is the disjoint union of $X$ and $Y$ 
and where the infimum is now taken over all
metrics on $Z$ which extend the metrics on $X$ and $Y$.

The Gromov-Hausdorff distance $d_{GH}$ defines a metric on the collection $\cC$ of all
compact metric spaces, considered up to isometry, and therefore
gives rise to the notion of {\it Gromov-Hausdorff convergence}
of sequences of compact metric spaces. Moreover, with respect to $d_{GH}$, the space
$\cC$ is complete.

\

It is sometimes also convenient to work 
with so-called (Gromov-) Hausdorff approximations:
If $X$ and $Y$ are compact metric spaces and $\eps$ is real and positive,
a (not necessarily continuous) map
$\phi: X\to Y$  is called an $\eps-${\it Hausdorff approximation from $X$ to $Y$}
if $\vert\ts d_Y(\phi (x), \phi(x')) - d_X(x, x') \ts\vert \ts\le \ts\eps$ for all $x, x'\in X$
and if the $\eps-$neighbourhood of $\phi (X)$ in $Y$ is equal to $Y$.

Notice that if $d_{GH}(X,Y)$ is less than $\eps$, then there exist
$3\eps-$Hausdorff approximations from $X$ to $Y$ and $Y$ to $X$. Conversely,
if there exist $\eps-$Hausdorff approximations between $X$ and $Y$,
then $d_{GH}(X,Y)\ts \le 3\eps$.
In particular, Gromov-Hausdorff convergence can therefore also
be defined via Hausdorff approximations.

\

There is also a
useful notion of Gromov-Hausdorff convergence for noncompact metric spaces, at
least for {\it length spaces} in which, by definition, the
distance between points is given by the infimum of the lengths of
all curves joining the points: 

A sequence $(X_n, x_n)_{n\in\N}$ of
locally compact length spaces with given basepoints $x_n\in X_n$ is said to converge
to a pointed metric space $(Y,y)$ in the pointed Gromov-Hausdorff
sense if, for all $r>0$, the closed $r$-balls $B_r(x_n)$
Gromov-Hausdorff converge in the usual sense to the closed r-ball $B_r(y)$ in $Y$.

Under this so-called {\it pointed Gromov-Hausdorff convergence}
or {\it Gro\-mov-Hausdorff convergence with basepoint},
the collection ${\cM}$  of all locally compact complete pointed length spaces, 
considered up to isometries that preserve basepoints,
is complete as well. Moreover, 
the same is true for ordinary Gromov-Hausdorff convergence 
in the subspaces ${\mathcal L} (D)$ of $\mathcal C$
which are made up of the isometry classes of compact length spaces
subject to a given uniform upper bound §$D$ for the diameter.

In the context of Riemannian manifolds 
with lower curvature bounds Gromov obtained
the following fundamental result:

\

\noindent
{\bf Theorem A.1 (Gromov's Precompactness Theorem)} \thinspace
{\it For any natural number $m$ and any real numbers $\kappa$ and $D$, 
the class of closed Riemannian $m$-manifolds with Ricci curvature 
${\rm Ric}$ $\ge (m-1)\ts \kappa$ and diameter
${\rm diam}$ $\le D$ is precompact in ${\mathcal L} (D)$ with respect to
 Gromov-Hausdorff convergence,
and the class of pointed complete Riemannian $m$-manifolds with $Ric\ge (m-1)\ts \kappa$
is precompact in $\mathcal M$ with respect to the pointed Gromov-Hausdorff topology.
}

\

If one replaces the Ricci curvature bounds in Theorem A.1 by sectional
curvature ones, one may actually replace ${\mathcal L} (D)$ and $\mathcal M$ 
by appropiate classes of Alexandrov spaces, but we will not need this here.
Instead, let us now turn to notions of Gromov-Hausdorff convergence 
for metric spaces which are equipped with isometric group actions.

\

Let ${\cM}$ be as above, and let ${\cM}^{Groups}$ consist of all triples $(X,\Gamma , p)$
with $(X,p) \in \cM$ such that $\Gamma$ is a closed subgroup
of the isometry group of $X$, and let in this case for real positive $r$ be
$\Gamma (r)$ the set of all $\gamma\in\Gamma$ for which $d(\gamma p, p) < r$.

\

\noindent
{\bf Definition A.2} \thinspace
{\it An $\eps$-{\rm equivariant Hausdorff approximation with basepoint}~$p$
from $(X,\Gamma , p) \in {\cM}^{G}$ to $ (Y, \Lambda, q) \in {\cM}^{Groups}$
is given by a triple of maps $(f,\phi,\psi)$ with

$$ f: B_p(1/\eps , X) \to Y, \quad \phi : \Gamma (1/\eps) \to \Lambda (1/\eps) 
\quad {\rm and} \quad \psi :  \Lambda (1/\eps) \to \Gamma (1/\eps)$$

that satisfies the following conditions:

\

1. $f(p)=q \ts ;$

\

2. the $\eps$-neighbourhood of $f(B_p(1/\eps , X))$ contains $B_q(1/\eps , Y)\ts ;$

\

3. for all $x_1, x_2 \in B_p(1/\eps , X)$ one has

$$ \vert\ts d(f(x_1) , f(x_2)) -  d(x_1 , x_2)\ts\vert\ts < \eps\ts ;$$

\

4. if $\gamma\in \Gamma (1/\eps)$, $x\in B_p(1/\eps , X)$, and $\gamma\ts x \in B_p(1/\eps , X)$,
then 
$$d\ts (f(\gamma\ts x) , \phi (\gamma) \ts (f(x))) \ts < \ts\eps \ts ;$$

\

5. if $\mu\in \Lambda (1/\eps)$, $x\in (B_p(1/\eps , X)$, and $\psi (\mu)\ts (x) \in (B_p(1/\eps , X)$,
then 
$$d \ts (f(\psi (\mu)\ts (x)) , \mu \ts (f(x)) \ts < \ts\eps \ts .$$
}

\

Notice that it is neither required that the maps occuring here
be globally defined, nor that they be continuous, nor, as regards $\phi$ and $\psi$,
be homomorphisms.

\

\noindent
{\bf Definition A.3} \thinspace
{\it Let $(X,\Gamma , p) , (Y, \Lambda, q) \in {\cM}^{Groups}$.
The  {\rm equivariant Gromov-Hausdorff distance with basepoint}
$d_{GH}^e ((X,\Gamma , p) , (Y, \Lambda, q))$ is defined as the
infimum of all positive real numbers $\eps$ such that there
exist $\eps$-{equivariant Hausdorff approximations with basepoint}
from 
$(X,\Gamma , p)$ to $(Y, \Lambda, q)$ and vice versa.

If $(X_n,\Gamma_n , p_n) , (Y, \Lambda, q) \in {\cM}^{Groups}$,
the equation $${\rm lim}_{n\to\infty} \quad (X_n,\Gamma_n , p_n) = (Y, \Lambda, q)$$
means that
$${\rm lim}_{n\to\infty} \quad d_{GH}^e ((X_n,\Gamma_n , p_n) , (Y, \Lambda, q)) = 0\ts .$$
}

\

If the groups that occur in the above definition are isomorphic to a fixed group $G$,
we will also speak of {\it $G$-equivariant Gromov-Hausdorff distance (with basepoint)}.

In the case where all groups are trivial, we will write $(X, p)$ instead
of $(X, {\rm id}, p)$ etc.~and $d_{GH}$ instead of $d_{GH}^e$,
and equivariant Gromov-Hausdorff
convergence with basepoint then reduces to ordinary Gromov-Hausdorff convergence with basepoint.

Let us also note that 
in the situtation where all spaces in question are compact and have uniformly
bounded diameters, Gromov-Hausdorff convergence and Gromov-Hausdorff convergence
with base point are essentially equivalent concepts: The latter clearly implies the former,
and up to passing to a subsequence the converse also holds.

\

Let us now gather several important facts and results concerning equivariant 
Gromov-Hausdorff convergence.

\

\noindent
{\bf Proposition A.4 (\cite{FY_92})} \thinspace
{\it If $(X_n,\Gamma_n , p_n)\in {\cM}^{Groups}$ and $(Y, q) \in {\cM}$ satisfy
with respect to the Gromov-Hausdorff distance with basepoint
$${\rm lim}_{n\to\infty} \quad (X_n, p_n) = (Y, q) \ts ,$$
then there exist $\Lambda$ and a subsequence $\{n_k\}_{k\in\N}$ such that
$(Y, \Lambda, q) \in {\cM}^{Groups}$ and such that 
$${\rm lim}_{k\to\infty} \quad (X_{n_k},\Gamma_{n_k} , p_{n_k}) = (Y, \Lambda, q) \ts .$$
}

\

Notice that even if all $\Gamma_n$ and $X_n$ above are the same,
in general the limit group $\Lambda$ can be a different group.
However, for compact groups and spaces we have the following stability result.

\

\noindent
{\bf Proposition A.5 (\cite[Theorem 6.9]{Fukaya_90})} \thinspace 
{\it Let $G$ be a group and $\mathcal M (G)$ denote the set
of all isometry classes of compact metric spaces equipped
with an isometric $G$-action. 
If $G$ is compact, then for any natural number $m$ 
and any real numbers $\kappa$ and $D$,
the class of closed $m$-dimensional Riemannian $G$-manifolds with Ricci curvature 
${\rm Ric}$ $\ge (m-1)\ts \kappa$ and diameter
${\rm diam}$ $\le D$ is precompact in ${\mathcal M (G)}$ with respect to the
$G$-equivariant Gromov-Hausdorff topology.
}

\

The following results establish several relations between equivariant 
Gromov-Hausdorff convergence and convergence of subgroups and quotient spaces.

\

\noindent
{\bf Proposition A.6 (\cite{Fukaya_86})} \thinspace
{\it 
If $(X_n,\Gamma_n , p_n) , (Y, \Lambda, q) \in {\cM}^{Groups}$ and
$${\rm lim}_{n\to\infty} \quad (X_n,\Gamma_n , p_n) = (Y, \Lambda, q) \ts ,$$
then
$${\rm lim}_{n\to\infty} \quad (X_n / \Gamma_n , \bar{p}_n) = (Y / \Lambda, \bar{q}) \ts ,$$
where $\bar{p}_n$ and $\bar{q}$ denote the images of ${p}_n$ and ${q}$, respectively,  under
the canonical orbit space projections $X_n\to X_n / \Gamma_n$
and $Y\to Y / \Lambda$.
}

\

\noindent
{\bf Proposition A.7 (\cite{Fukaya_86})} \thinspace
{\it Let $X_n$ and $Y$ be complete and simply connected Riemannian manifolds
and $\Gamma_n$ resp. $\Lambda$ be closed subgroups of the isometry groups
of $X_n$ resp. $Y$ which all act effectively and properly discontinuously.
Suppose, moreover, that the sectional curvatures of $X_n$ and $Y$ are
uniformly bounded in absolute value. Then 
$${\rm lim}_{n\to\infty} \quad (X_n / \Gamma_n , \bar{p}_n) = (Y / \Lambda, \bar{q}) \ts $$
implies 
$${\rm lim}_{n\to\infty} \quad (X_n,\Gamma_n , p_n) = (Y, \Lambda, q) \ts .$$
}

\

\noindent
{\bf Proposition A.8 (\cite{FY_92}) } \thinspace
{\it 
Let $(X_n,\Gamma_n , p_n) , (Y, \Lambda, q) \in {\cM}^{Groups}$ satisfy
$${\rm lim}_{n\to\infty} \quad (X_n,\Gamma_n , p_n) = (Y, \Lambda, q) \ts $$
and let $\Lambda'$ be a normal subgroup of $\Lambda$ such that the
following conditions hold:

1. $\Lambda / \Lambda'$ is discrete and finitely presented;

2. $Y/\Lambda$ is compact;

3. there exists $r>0$ such that $\Lambda'$ is generated by $\Lambda'(r)$ and such that 
the homomorphism $\pi_1(B_q(r,Y), q) \to \pi_1(Y,q)$ is surjective;

4. for all $n$, $X_n$ is simply connected and $\Gamma_n$ acts freely and properly discontinuously on $X_n$.

\

Then there exists a sequence of normal subgroups $\Gamma'_n$ of $\Gamma_n$ satisfying

$$(a) \quad {\rm lim}_{n\to\infty} \quad (X_n , \Gamma'_n , {p}_n) = (Y ,\Lambda', {q}) \ts ;$$

$$(b) \quad \Gamma_n / \Gamma'_n \text{ is isomorphic to } \Lambda / \Lambda'
\text{ for all sufficiently large } n \ts .$$
}

\

\section{Collapsing of manifolds under both-sided bounds 
on sectional curvature}

We will now briefly describe the collapsing  
of Riemannian manifolds with
both-sided bounds on sectional curvature, where a detailed structure theory 
has been initiated and
developed in the works of Cheeger, Fukaya, Gromov, and others.

\

Let $M=(M^m,g)$ be a Riemannian manifold of dimension $m$    
and let $FM=F(M^m)$ denote its bundle of orthonormal frames. 
When fixing a bi-invariant metric on $O(m)$,    
the Levi-Civita connection of $g$ gives rise to a canonical metric on 
$FM$, so that the projection $FM\to M$ becomes a Riemannian submersion 
and so that $O(m)$ acts on $FM$ by isometries.   
Another fibration structure on $FM$ 
is called {\it $O(m)$ invariant}
if the $O(m)$ action on $FM$    
preserves both its fibres and its structure group.    
   
A {\it pure N-structure} on $M^m$ is defined    
by an $O(m)$ invariant fibration, $\tilde\eta: FM\to B$, with     
fibre a nilmanifold isomorphic to $(N/\Gamma,   
\nabla^{\text{can}})$ and structural group contained in the group   
of affine automorphisms of the fibre, where 
$B$ is a smooth manifold, $N$ is    
a simply connected nilpotent group and $\nabla^{\text{can}}$ the   
canonical connection on $N$ for which all left invariant vector   
fields are parallel. 
A pure $N$-structure on $M$ induces, by $O(m)$-invariance, a partition of
$M$ into ``orbits'' of this structure, and is then said
to have {\it positive rank}
if all these orbits have positive dimension.
A pure N-structure $\tilde\eta: FM\to B$ over    
a Riemannian manifold $(M,g)$   
gives rise to a sheaf on $FM$ whose local sections   
restrict to local right invariant vector fields on the fibres   
of $\tilde\eta$, and if the local sections of this sheaf are local Killing fields   
for the metric $g$, then $g$ is said to be {\it invariant}    
for the N-structure (and  $\tilde\eta$ is then also sometimes referred to as 
{\it pure nilpotent Killing structure for} $g$).
Cheeger, Fukaya and Gromov showed:

\

\noindent
{\bf Theorem B.1 (\cite{CFG}) } \thinspace
{\it
Let for $m\ge 2$ and $D>0$ \ts    
${\frak M}(m,D)$ denote the class of all m-dimensional   
compact connected Riemannian manifolds $(M,g)$ with sectional    
curvature $\vert K_g\vert \le 1$ and diameter $diam(g)\le D$.    
 
Then, given any $\eps>0$,  
there exists a positive number $v=v(m,D,\eps)>0$    
such that if $(M,g)\in{\frak M}(m,D)$ satisfies \ts $vol(g)<v$, then 
$M^m$ admits a pure N-structure $\tilde\eta: FM\to B$ of positive rank so that 

\

\noindent
{(a)} there is a smooth metric $g_\eps$ on $M$ which is invariant for 
the N-structure $\tilde\eta$ and for which all fibres of $\tilde\eta$ have 
diameter less than $\eps$, satisfying 
$$e^{-\eps} g < g_\eps < e^{\eps} g \ , \quad\quad   
  \mid {\nabla}_g - {\nabla}_{g_\eps} \mid \ < \ \eps \ , \quad\quad   
  \mid {\nabla}_{g_\eps}^l R_{g_\eps} \mid \ < \ C(m,l,\eps) \ ;$$

\noindent
{(b)} there exist constants $i=i(m,\eps)>0$ and $C=C(m,\eps)$    
           such that, when equipped with the metric induced by $g_\eps$, 
           the injectivity radius of $B$ is $\ge i$    
           and the second fundamental form of all fibres of    
           $\tilde\eta$ is bounded by $C$.   
}

\

The $O(m)$ invariance of a pure N-structure $\tilde\eta: FM\to B$  
implies that the $O(m)$ action on   
$FM$ descends to an $O(m)$ action on $B$ and that the fibration    
on $FM$ descends to a possibly singular fibration on $M$,    
$\eta: M^m\to B/O(m)$, such that the following diagram commutes.   
   
$$\CD F(M^m) @>\tilde \eta>> B\\   
@VV \pi V   @VV \tilde \pi V\\   
M^m @> \eta >> B/O(m)    
\endCD$$

\

Theorem~B.1 describes the structure of a general  collapse with bounded curvature,
where $B/O(m)$ is in general not a manifold.
However,
in the case where one has a bounded curvature collapse to a manifold limit space,
there is also Fukaya's following equivariant fibre bundle version of the above result:

\

\noindent
{\bf Theorem B.2 (\cite{Fukaya_90, Fukaya_89, Fukaya_88b}) } \thinspace
{\it Let $M_i \in {\frak M}(m,D)$,  $N \in {\frak M}(n,D)$ and let $G$ be a compact Lie group
which acts on $M_i$ and $N$ by isometries so that the $M_n$
converge for $i\to\infty$ to $N$ in the $G$-equivariant Gromov-Hausdorff distance.
Then, for $i$ sufficiently large, there exist  mappings
$\pi_i : M_i \to N$ which satisfisfy

\

\noindent
{(a)} $\pi_i : M_i \to N$ is a fibre bundle and $\pi_i$ is $G$-equivariant;

\

\noindent
{(b)} each fiber $\pi_{i}^{-1}(p)$ carries a flat connection 
which depends smooth\-ly on $p \in N$, such that $\pi_{i}^{-1}(p)$ is
affinely diffeomorphic to an infranil manifold 
$N_{i} /\Delta_{i}$ with its canonical flat connection $\ncan$;

\

\noindent
{(c)} the affine structures on the fibres of $\pi_i : M_i \to N$ are $G$-invariant;

\

\noindent
{(d)} the structure group of $\pi_{i}$ is contained in the
Lie group
$$  { \C(N_{i}) \over \C(N_{i}) \cap \Delta_{i} }  \rtimes {\rm Aut}(\Delta_{i})\,  ,
$$ where $\C(N_{i})$ denotes the center of $N_{i}$;

\

\noindent
{(e)} $\pi_{i}$ is an almost Riemannian submersion in the sense that for any $p\in N$ and
any tangent vector $v\in T_p M_i$ orthogonal to the fibre $\pi_{i}^{-1}(p)$,
one has an estimate
$$e^{-o(i)}\quad < \quad  \vert\vert d\pi_i (v) \vert\vert \ts\ts  / \ts\ts \vert\vert v \vert\vert 
\quad < \quad e^{o(i)} \ts ,$$
where $o(i)$ satisfies ${\text{lim}_{i\to\infty} o(i) = 0}$.
}

\

\end{document}